\documentclass[11pt,a4paper]{article}

\usepackage{mathrsfs}
\usepackage{amscd}
\usepackage{hyperref}
\hypersetup{colorlinks=true,citecolor=green,linkcolor=blue,urlcolor=blue}
\usepackage{tipa}
\usepackage{amssymb}
\usepackage{amsfonts}
\usepackage{stmaryrd}
\usepackage{amssymb}

\usepackage{CJK}  
\usepackage{indentfirst} 

\usepackage{latexsym}   
\usepackage{bm}         

\usepackage{pifont}
\usepackage{bm}
\usepackage{amsmath,amssymb,amsfonts}
\usepackage{indentfirst}

\usepackage{xcolor}
\usepackage{cases}
\usepackage{wasysym}
\usepackage{amsthm}

\newcommand{\hei}{\CJKfamily{hei}}      
\newtheorem{theorem}{Theorem}[section]
\newtheorem{lemma}[theorem]{Lemma}
\newtheorem{proposition}[theorem]{Proposition}

\newtheorem{remark}[theorem]{Remark}

\CJKtilde   

\begin{document}

\title{\huge \hei A remark on the simple cuspidal representations of $GL(n, F)$}
\date{}
\author{\hei Peng Xu}
\maketitle

\begin{abstract}
Let $F$ be a non-archimedean local field of residue characteristic $p$, and let $G$ be the group $GL(n, F)$. In this note, under the assumption $(n, p)=1$, we show that a simple cuspidal representation $\pi$ (that is with normalized level $\frac{1}{n}$) of $G$ is determined uniquely up to isomorphism by the local constants of $\chi\circ \text{det}\otimes \pi$, for all characters $\chi$ of $F^\times$.

\end{abstract}

\begin{section}{Introduction}

Let $F$ be a non-archimedean local field with integer ring $\mathfrak{o}_F$ and maximal ideal $\mathfrak{p}_F$ and assume its residue field $k_F= \mathfrak{o}_F/\mathfrak{p}_F$ is of order $q$ and of characteristic $p$. Fix a prime element $\varpi$ and a root of unity $\eta$ of order $q-1$ in $F$. Let $G$ be the general linear group $GL(n, F)$. In this short note, we investigate some aspect of simple cuspidal representations of $G$, especially the behaviour of their local constants under twists by characters of $F^{\times}$. The main result is the following Theorem \ref{main result}, which in particular verifies a very special case of Jacquet's conjecture on the local converse theorem of $G$ (\cite{CPS1994}).

We fix a level one additive character $\psi$ (i.e., $\psi$ is trivial on $\mathfrak{p}_F$ but non-trivial on $\mathfrak{o}_F$) of $F$.

In this note, for a cuspidal representation $\pi$ of $G$ to be simple, we mean it has normalized level $l(\pi)=\frac{1}{n}$. Denote $\chi\circ\textnormal{det}\otimes \pi$ by $\chi\pi$ as usual.

\begin{theorem}\label{main result}
Assume $(n, p)=1$. Let $\pi_1$ and $\pi_2$ be two cuspidal representations of $G$, such that
\begin{equation}\label{basic assumption}
\varepsilon(\chi\pi_1, s, \psi)=\varepsilon(\chi\pi_2, s, \psi),
\end{equation}
for all characters $\chi$ of $F^\times$ and $s\in \mathbb{C}$, which forces $\pi_1$ and $\pi_2$ to have the same normalized level $l$. If further $l=\frac{1}{n}$, then
\begin{center}
$\pi_1 \cong \pi_2.$
\end{center}
\end{theorem}

\begin{remark}
The tameness condition $(n, p)=1$ is crucially used in the argument, but it is reasonable to believe the result should hold without it.
\end{remark}

\begin{remark}
In a recent preprint \cite{AL13}, $\textnormal{Moshe~Adrian}$ and $\textnormal{Baiying~Liu}$ have also obtained the same result as Theorem \ref{main result}, via a different method.
\end{remark}

\end{section}

\begin{section}{Preliminary facts}
In this section, we recall some well-known facts, for which we also include a sketched proof and detailed references.
\begin{proposition}\label{initial pro}
Let $\pi_1$ and $\pi_2$ be two cuspidal representations of $G$, such that
\begin{equation}\label{initial condition}
\varepsilon(\chi\pi_1, s, \psi)=\varepsilon(\chi\pi_2, s, \psi),
\end{equation}
for all characters $\chi$ of $F^\times$ and $s\in \mathbb{C}$. Then

$(\textnormal{i})$~~The identity \eqref{initial condition} holds when $\psi$ is replaced by any additive character of $F$.

$(\textnormal{ii})$~~$\pi_1$ and $\pi_2$ have the same central characters, i.e., $\omega_{\pi_1}=\omega_{\pi_2}.$

$(\textnormal{iii})$~~$\pi_1$ and $\pi_2$ have the same normalized level, i.e., $l(\pi_1)=l(\pi_2)$.
\end{proposition}

\begin{proof}
$\textnormal{(i)}$ is direct from the definition, combined with \eqref{initial condition}. For an irreducible cuspidal representation of $G$, the identity
\begin{equation}\label{s to half lemma}
\varepsilon(\pi, s, \psi)= q^{n\cdot l(\pi)(\frac{1}{2}-s)}\varepsilon(\pi,\frac{1}{2} , \psi)
\end{equation}
holds (see $6.1.2$ in \cite{BHK1998} and note that $\psi$ is chosen to be level one), from which $(\textnormal{iii})$ follows.

$(\textnormal{ii})$ follows from the following Lemma, as in $27.4$ of \cite{BH2006}.
\begin{lemma}\label{stability}
Let $\pi$ be a cuspidal representation of $G$ and let $\chi$ be a character of $F^{\times}$, such that,
\begin{center}
$m=l(\chi)> 2 l(\pi)$,
\end{center}
where $l(\chi)$ is the level of $\chi$. Let $c$ be an element in $F^\times$ such that $\chi(1+x)= \psi (c\cdot x)$ for $x\in \mathfrak{p}^{[\frac{m}{2}]+1}$, then
\begin{center}
$\varepsilon(\chi\pi, s, \psi)=\omega_{\pi}(c)^{-1}\varepsilon(\chi\circ \textnormal{det}, s, \psi)$.
\end{center}

\end{lemma}

\begin{proof}
This is a minor refinement of a Lemma of Jacquet-Shalika \cite{JS1985}. We include a detailed proof in Appendix \ref{appendix} for the reader's convenience, following \cite{BH2006}.
\end{proof}

This completes the proof of Proposition \ref{initial pro}.
\end{proof}

\medskip

Let $\mathbb{P}_{n}(F)$ be the set of isomorphism classes of admissible pairs of degree $n$, and $\mathcal{A}_{n}^{et}(F)$ be the set of isomorphism classes of essentially tame cuspidal representations of $G$. For the exact definitions of admissible pairs and essentially tame cuspidal representations, see \cite{BH2005}.
\begin{theorem}\label{b-h}
There is a natural bijection between $\mathbb{P}_{n}(F)$ and $\mathcal{A}_{n}^{et}(F)$.
\end{theorem}

\begin{proof}
Theorem $2.3$, \cite{BH2005}.
\end{proof}

Denote by $\pi_{E, \theta}$ the cuspidal representation which arises from an admissible pair $(E/F,~ \theta)$, via the above theorem. The content we need from Theorem $2.3$ of \cite{BH2005} is summarized in the following proposition:

\begin{proposition}\label{property from b-h}
$(\textnormal{i})$~~$l(\pi_{E, \theta})=\frac{l(\theta)}{e(E/F)}$, where $e(E/F)$ is the ramification index of $E/F$.

$(\textnormal{ii})$~~$\omega_{\pi_{E, \theta}}=\theta|_{F^\times}.$

$(\textnormal{iii})$~~$\chi\cdot \pi_{E, \theta}\cong \pi_{E, \chi_{E}\cdot\theta}$, where $\chi_{E}= \chi\circ N_{E/F}$.
\end{proposition}

\begin{proof}
 $\textnormal{(ii)}$ and $\textnormal{(iii)}$ are the contents of Proposition $2.4$, \cite{BH2005}. $\textnormal{(i)}$ can be easily concluded from the constructions in $2.3$, \cite{BH2005}.\footnote{To be consistent, we use the notations of 2.3 of \cite{BH2005}. In 2.3 of \cite{BH2005}, the cuspidal representation $_F \pi _{\xi}$ constructed from an admissible pair $(E/F, \xi)$ contains some simple character arising from a simple stratum $[\mathfrak{A},l, 0, \beta]$, where $l= l(\phi)$, and $\phi$ is a character of $E'^{\times}$ chosen to satisfy $\xi\mid U^{1}_E= \phi\circ N_{E/E'}$. As $E/E'$ is unramified, one has $l(\phi)= l(\xi) (\geq 1)$. Recall the definition of normalized level, the construction in 2.3 of \cite{BH2005} implies $l(_F \pi _{\xi})= \frac{l}{e(E'/F)}$ .}
\end{proof}

For the argument in Section \ref{argument of main result}, we need to recall the local constant of a simple cuspidal representation $\pi$ arising from an admissible pair $(E/F, \theta)$.

We assume $(n, p)=1$ in the following Proposition.
\begin{proposition}\label{local cuspidal of a simple pi}
Let $\pi_{E, \theta}$ be a cuspidal representation arising from an admissible pair $(E/F, \theta)$, where $E/F$ is a totally ramified extension of degree $n$ (hence $(n, p)=1$) and $\theta$ is a character of $E^\times$ and of level $2k+1$ for some $k\geq 0$. Choose  $\alpha\in \mathfrak{p}^{-(2k+1)}_E/\mathfrak{p}^{-k}_E$, such that $\theta(1+x)= \psi_{E/F}(\alpha x)$ for $x\in \mathfrak{p}^{k+1}_{E}$, where $\psi_{E/F}$ is $\psi\circ \text{tr}_{E/F}$. Then,
\begin{center}
$\varepsilon(\pi_{E, \theta}, \frac{1}{2}, \psi)= \theta(\alpha)^{-1}\psi_{E/F}(\alpha).$
\end{center}

\end{proposition}

\begin{proof}
This follows directly from Proposition $1$ of $6.3$ in \cite{BH1999} and the construction of $\pi_{E, \theta}$ from an admissible pair $(E/F, \theta)$ ($2.3$, \cite{BH2005}).
\end{proof}

\end{section}

\smallskip

\begin{section}{Proof of Theorem 1.1}\label{argument of main result}

The strategy of using admissible pairs was suggested by Professor Guy Henniart.

\begin{subsection}{Reduction to the same field extension}

From now on, we assume $(n, p)=1$.

Let $\pi_1$ and $\pi_2$ be two cuspidal representations of $G$ of level $\frac{2k+1}{n}$, where $k\geq 0$ and $(n, 2k+1)=1$, which satisfy $\omega_{\pi_1}= \omega_{\pi_2}$. For $i=1, 2$, assume $(E_i /F,  \theta_i)$ is an admissible pair associated to $\pi_i$ via Theorem \ref{b-h}. Then from Proposition \ref{property from b-h}, $E_i/F$ is a totally ramified extension of degree $n$ and $\theta_i$ is a character of $E^{\times}_i$ of level $2k+1$. Also, from Proposition \ref{property from b-h}, the restrictions of $\theta_1$ and $\theta_2$ to $F^{\times}$ coincide.

\begin{proposition}\label{reduce to the same field}

Assume $\pi_1$ and $\pi_2$ satisfy the condition \eqref{basic assumption} in Theorem \ref{main result}. Then $E_1$ is isomorphic to $E_2$ over $F$.

\end{proposition}

\begin{proof}

From Chapter $16$ of \cite{Has1980}, there are in all $e= (n, q-1)$ different totally tamely ramified extension of degree $n$ over $F$, which can be described as:
\begin{center}
$F(\sqrt[n]{\varpi \eta^{r}}), ~0\leq r< e$.
\end{center}
Hence we assume $e> 1$. Assume $E_1$ and $E_2$ are different over $F$. Without loss of generality, we may then assume further that
\begin{center}
$E_1=F(\sqrt[n]{\varpi}),~E_2= F(\sqrt[n]{\varpi \eta^a}), \text{where}~ 0< a<e$.
\end{center}
We note $\varpi_1= \sqrt[n]{\varpi}$ and $\varpi_2=\sqrt[n]{\varpi \eta^a}$ are respectively prime elements in $E_1$ and $E_2$.

 Write $2k+1= a'n+b$,  for $a'\geq 0,~0< b\leq n-1$. Note that $b$ is coprime to $n$.  The additive character $\psi_{E_i}= \psi\circ \text{tr}_{E_{i} /F}$ of $E_i$ is also of level one, as $E_{i}/F$ is a tame extension and $\psi$ is of level one.

As $\theta_i$ is of level $2k+1$, there is a unique $\alpha_i+ \mathfrak{p}^{-k}_{E_i}\in \mathfrak{p}^{-(2k+1)}_{E_i}/\mathfrak{p}^{-k}_{E_i}$ such that
\begin{center}
$\theta_i (1+x)= \psi_{E_i/F}(\alpha_i x),$ for $x\in \mathfrak{p}^{k+1}_{E_i}$.
\end{center}
Write $\alpha_i$ as $\varpi^{-a'}\varpi^{-b}_i \eta^{a_i}\beta_i$, for some $0\leq a_i <q$ and some $\beta_i\in U^{1}_{E_i}= 1+\mathfrak{p}_{E_i}$. Then, using the assumption on the local constants of $\pi_i$ for twists by level zero characters $\chi$ of $F^{\times}$, we are given a family of identities from Proposition \ref{local cuspidal of a simple pi}
\begin{equation}\label{second identity}
\theta_1\cdot\chi_{E_1}(\alpha_1)^{-1}\psi_{E_1 /F}(\alpha_1)=\theta_2\cdot\chi_{E_2}(\alpha_2)^{-1}\psi_{E_2 /F}(\alpha_2).
\end{equation}
We emphasize that $\chi$ is chosen to be of level zero, which is the reason that one can still use $\alpha_i$ in both sides of \eqref{second identity}.

Then from \eqref{second identity}, we get
\begin{equation}\label{third equation}
\chi(\eta^{n(a_1- a_2)-ab})=\theta_1 (\alpha_1)^{-1}\psi_{E_1/F} (\alpha_1)\psi_{E_2 /F}(-\alpha_2)\theta_2 (\alpha_2).
\end{equation}

The left hand side of \eqref{third equation} cannot be  constant when $\chi$ goes through all the level zero characters of $F^{\times}$, as under our assumption $q-1$ does not divide $n(a_1- a_2)-ab$. We get a contradiction.
\end{proof}

\end{subsection}

\begin{subsection}{The case of simple cuspidal representations ($k=0$)}

With the same assumptions as in the last subsection, we carry on to prove the admissible pairs of $\pi_1$ and $\pi_2$ are isomorphic over $F$ when $k=0$. Hence we prove in this case that $\pi_1\cong \pi_2$ by Theorem \ref{b-h}.

From Proposition \ref{reduce to the same field}, one can take $E_1=E_2=E= F(\sqrt[n]{\varpi})$. Denote $\sqrt[n]{\varpi}$ by $\varpi_E$.

Now we repeat a bit more from the last section. Choose $\alpha_i+ \mathfrak{o}_E\in \mathfrak{p}^{-1}_E/\mathfrak{o}_E$, such that
\begin{center}
$\theta_i (1+x)= \psi_{E/F}(\alpha_i x),$ for $x\in \mathfrak{p}_{E}$.
\end{center}
We note the choice of $\alpha_i$ is up to multiplication by $U^{1}_E$. In writing $\alpha_i$ as $\varpi^{-1}_E \eta^{a_i}\beta_i$, for some $0\leq a_i <q$ and some $\beta_i\in U^{1}_{E}= 1+\mathfrak{p}_{E}$, we may assume $\beta_i= 1$. Also we know $\text{tr}_{E/F}(\varpi^{c}_E)=0$ when $n\nmid c$. Hence,
\begin{center}
$\psi\circ \text{tr}_{E/F} (\alpha_i)=1.$
\end{center}

In all, we get a simplified version of \eqref{second identity}
\begin{center}
$\chi (\eta)^{n(a_2-a_1)}= \theta_1 (\varpi_E)^{-1}\theta_2 (\varpi_E)\xi_{\eta}^{a_1-a_2},$
\end{center}
where $\xi_{\eta} =\theta_1 (\eta)=\theta_2 (\eta)$. The left hand side of the above equation is constant for all $\chi$ of level zero, only if $q-1$ divides $n(a_2-a_1)$; as a result $\eta^{a_1-a_2}$ is an $n$-th root of unity in $F$.

Denote by $\sigma$ the automorphism of $E$ over $F$, determined by sending $\varpi_E$ to $\varpi_E \cdot \eta^{a_1-a_2}$ (which is a conjugate of $\varpi_E$). Then one can easily check
\begin{center}
$\theta_1= \theta_2 \circ \sigma$.
\end{center}

This complete the proof of Theorem \ref{main result}.
\end{subsection}

\begin{remark}
As we have seen, under the assumption $k=0$ the situation is essentially simplified, which makes the final argument completely elementary. However, once $k$ becomes larger than zero, it is not clear (to the author) what one should expect for the relations between $\theta_1$ and $\theta_2$, even involving $\chi$ of levels bigger than zero.
\end{remark}

\end{section}

\begin{section}{Appendix A: proof of Lemma \ref{stability}}\label{appendix}
In this appendix we carry out the proof of Lemma \ref{stability}, following the process in 25.7 of \cite{BH2006}. The only difference here is that we include some details on the local constant $\varepsilon(\chi\circ \text{det}, s, \psi)$ of the one-dimensional character $\chi\circ\text{det}$ of $G=GL(n, F)$, for a character $\chi$ of $F^\times$ of level $l\geq 1$. When $n=2$, it is indeed an exercise in the excellent book \cite{BH2006}. On the one hand, we have the following first:
\begin{lemma}\label{1=n lemma}
$\varepsilon(\chi\circ \textnormal{det}, s, \psi)=\varepsilon(\chi, s, \psi)^n.$
\end{lemma}
\begin{proof}
By writing $\chi \circ \text{det}$ as $Q(\chi\cdot|\cdot|_F ^{n-1}, \ldots, \chi\cdot|\cdot|_F ^{1-n})$ in the Langlands classification,  the Lemma is a special case of 3.1.4 in \cite{Kudla1991}.
\end{proof}

 Choose any principal hereditary order $\mathfrak{A}$ in $A=M_n (F)$ of ramification index $e_\mathfrak{A}$, with Jacobson radical $\mathfrak{P}$. Then the restriction of $\chi\circ\text{det}$ to $\mathcal{K}_{\mathfrak{A}}$ is of level $e_{\mathfrak{A}}l$, where $\mathcal{K}_{\mathfrak{A}}$ is the normalizer of $\mathfrak{A}$ in $G$. Choose $c\in \mathfrak{p}^{-l}$, such that $\chi(1+x)= \psi (cx)$ for $x\in \mathfrak{p}^{[l/2]+1}$. Then one may check directly that
 \begin{center}
 $\chi\circ \text{det}\mid U^{[e_{\mathfrak{A}}l/2]+1}_{\mathfrak{A}}= \psi_c,$
 \end{center}
where $\psi_c$ is the additive character on $U^{[e_{\mathfrak{A}}l/2]+1}_{\mathfrak{A}}$: $\psi_c (1+y)= \psi\circ \text{tr}_A (cy)$, for $y\in \mathfrak{P}^{[e_{\mathfrak{A}}l/2]+1}$.

\begin{lemma}\label{local constants of character}
One has
\begin{equation}\label{local constants identity for character}
\varepsilon(\chi\circ \textnormal{det}, s, \psi)= q^{nl(\frac{1}{2}-s)}(\mathfrak{A}: \mathfrak{P}^{e_{\mathfrak{A}}l+1})^{-\frac{1}{2}}\cdot \tau_{\mathfrak{A}}(\chi, \psi),
\end{equation}
where $\tau_{\mathfrak{A}}(\chi, \psi)$ is the Gauss sum defined as follows,
\begin{equation}\label{definition of tau}
\tau_{\mathfrak{A}}(\chi, \psi)= \sum_{y\in U_\mathfrak{A}/ U^{e_\mathfrak{A}l+1}_{\mathfrak{A}}} \chi^{-1}(\textnormal{det} (cy))\psi_A (cy)
\end{equation}
which simplifies to
\begin{equation}\label{simplified version of tau}
\tau_{\mathfrak{A}}(\chi, \psi)= (\mathfrak{A}: \mathfrak{P}^{[(e_{\mathfrak{A}}l+1)/2]})\sum_{y}  \chi^{-1}(\textnormal{det} (cy))\psi_A (cy),
\end{equation}
where $y$ goes through $U^{[(e_{\mathfrak{A}}l+1)/2]}_{\mathfrak{A}}/U^{[e_{\mathfrak{A}}l/2]+1}_{\mathfrak{A}}.$

\end{lemma}

\begin{proof}
By the remarks proceeding the Lemma, it is purely formal (and standard) to arrive at \eqref{simplified version of tau} from \eqref{definition of tau} .

We first simplify the RHS of \eqref{local constants identity for character}.

Write $e_{\mathfrak{A}}$ as $e$ for short. In fact, the following identity is well-known, although its proof is scattered in the literature:
\begin{equation}\label{sum=power of sum}
(\mathfrak{A}: \mathfrak{P}^{el+1})^{-\frac{1}{2}}\cdot \tau_{\mathfrak{A}}(\chi, \psi)= q^{-n(l+1)/2} \tau(\chi, \psi)^n
\end{equation}
where $\tau(\chi, \psi)$ is the classical Gauss sum in Tate's thesis (23.6.4 of \cite{BH2006}).

From 1.8 in \cite{Bus1987}, the index of $\mathfrak{P}$ in $\mathfrak{A}$ is $q^{n^2 /e}$. Hence, it suffices to simplify the sum appearing in $\tau_{\mathfrak{A}}(\chi, \psi)$. Denote respectively by $c', c''$ the integers $[(el+1)/2]$ and $[el/2]+1$. When $el+1$ is even (hence $l$ is odd), the sums in both sides of \eqref{sum=power of sum} become one term, and one can check the equation holds immediately, by taking $y=Id$. We assume $el+1$ is odd, i.e., $2 \mid el$.

We check the case $l=2m$ in detail; where the situation when $e$ is even and $l$ is odd follows in the same manner. Clearly, one has $c'=em$, $c''= em+1$. For $y\in U^{c'}_{\mathfrak{A}}/U^{c''}_{\mathfrak{A}}$, we write $y=1+ \varpi^{m}a$, for some $a= (a_{ij})_{1\leq i, j \leq n}\in \mathfrak{A}$. From the description of $\mathfrak{A}$ in (2.5) of \cite{BK1993} as an $e\times e$-block matrix, we see
\begin{equation}\label{expansion of det}
\text{det}(y)=\prod_{1 \leq i \leq n} (1+ \varpi^{m}a_{ii})+\sum^{e}_{s=1}\sum_{i< j} u^{s}_{i, j}a_{ij}a_{ji}\varpi^{2m} + \sum (\text{remaining~terms}),
\end{equation}
where the second inner sum runs through all the integer pairs $(i, j)$ in $((s-1)n/e, sn/e]$, and $u^{s}_{i, j}$ is some unit in $\mathfrak{o}^{\times}_F$. We note that there are in all $e\cdot \frac{n}{e}(\frac{n}{e}-1)/2= \frac{n^2- ne}{2e}$ terms in the second sum of \eqref{expansion of det}. Note also that the terms in the third sum of \eqref{expansion of det} will be killed by $\chi$, as $\chi$ is of level $l=2m$. We are now able to verify \eqref{sum=power of sum} easily:
\begin{equation}
 \underset{y}{\sum}~~\chi^{-1}(\textnormal{det}cy)\psi_{A}(cy)=t
\prod_{1\leq i \leq n}\underset{a_{ii}\in \mathfrak{o}/\mathfrak{p}
}{\sum}~~\chi^{-1}(c(1+\varpi^{m}a_{ii})) \psi(c(1+\varpi^{m}a_{ii})),
\end{equation}
where $t$ is the following quantity:
\begin{center}
$t= \prod_{1\leq s \leq e,~ (s-1)n/e< i< j\leq  sn/e}\sum_{a_{ij}, a_{ji}\in \mathfrak{o}/\mathfrak{p}}ua_{ij}a_{ji}$.
\end{center}
The following easy identity shows that the value of $t$ is $q^{\frac{n^2- ne}{2e}}$, which completes the proof of \eqref{sum=power of sum} in the case that $l$ is even: for any unit $u\in \mathfrak{o}^{\times}_F$, one has

\begin{equation}
\sum_{a, ~b\in \mathfrak{o}/\mathfrak{p}}\psi(uab)=q,
\end{equation}

We have indeed verified that the RHS of \eqref{local constants identity for character} does not depend on the choice of $\mathfrak{A}$. \eqref{local constants identity for character} is reduced to the following:
\begin{equation}\label{final identity of local constants of characters}
\varepsilon(\chi\circ \textnormal{det}, s, \psi)= q^{nl(\frac{1}{2}-s)}q^{-n(l+1)/2} \tau(\chi, \psi)^n
\end{equation}
Note the RHS of \eqref{final identity of local constants of characters} is $(q^{l(\frac{1}{2}-s)}q^{-(l+1)/2}\tau(\chi, \psi))^{n}$, which is just $\varepsilon(\chi, s, \psi)^n$ by 23.6.2 of \cite{BH2006}. We are done, by Lemma \ref{1=n lemma}.
\end{proof}

We now complete the proof of Lemma \ref{stability}; actually based on Lemma \ref{local constants of character} this is just to repeat the argument in 25.7 of \cite{BH2006}. We will use the language of \cite{BH1999} freely.

Let $\Lambda$ be a central type contained in $\pi$, say $\Lambda\in \mathcal{C}\mathcal{C}(\mathfrak{A}, \beta)$, for some principal hereditary order $\mathfrak{A}$ and some element $\beta\in \mathfrak{A}$. Then the level $l(\Lambda)$ of $\Lambda$ is $e_{\mathfrak{A}}\cdot l(\pi)$, where $e=e_{\mathfrak{A}}$ is the ramification index of $\mathfrak{A}$. Then, one has
\begin{equation}\label{local constants of cus}
\varepsilon(\pi, \frac{1}{2}, \psi)= (\mathfrak{A}: \mathfrak{P}^{1+l(\Lambda)})^{-\frac{1}{2}}\tau (\Lambda, \psi),
\end{equation}
where $\tau (\Lambda, \psi)$ is the Gauss sum defined in \cite{BH1999}, and can be simplified as:
\begin{equation}\label{gauss sum for local cuspidal}
\tau(\Lambda, \psi)=c_1 \underset{y\in
U_{\mathfrak{A}}^{[(l(\Lambda)+1)/2]}/U_{\mathfrak{A}}^{[l(\Lambda)/2]+1}}{\sum}~~\text{tr}\Lambda^{\vee}
(\beta y)\psi_{A}(\beta y),
\end{equation}
in which $c_1= \frac{(U_\mathfrak{A} ^{[l(\Lambda)/2]+1}: U_\mathfrak{A} ^{l(\Lambda)+1})}{\text{dim}\Lambda}$.

Now for a character $\chi$ of level $m> 2l(\pi)$,  the cuspidal representation $\chi \pi$ contains the central type $\chi\Lambda$, which is of level $em$. More precisely, $\chi\circ\text{det}\otimes \Lambda\in \mathcal{C}\mathcal{C}(\mathfrak{A}, c+\beta)$\footnote{One indeed needs to check that the group $\mathbf{J}_{\beta+c}$ arising from the simple stratum $(\mathfrak{A}, em, 0, \beta+c)$ coincides with the $\mathbf{J}_\beta$ arising from $(\mathfrak{A}, l(\Lambda), 0, \beta)$, but this is directly from \cite{BK1993}.}. As $m> 2 l(\pi)$, $\Lambda$ is trivial on $U^{[(em+1)/2]}_{\mathfrak{A}}$,  and hence $\chi\circ\text{det}\otimes \Lambda \mid U^{[(em+1)/2]}_{\mathfrak{A}}=\chi\circ\text{det}.$
The identity in Lemma \ref{stability} follows by using \eqref{s to half lemma}, \eqref{local constants of cus}, \eqref{gauss sum for local cuspidal}, combing Lemma \ref{local constants of character} (note that $1+c^{-1}\beta\in \mathfrak{P}^{[em/2]+1}$ under the assumption). We are done.

\end{section}

\section*{Acknowledgements}
The work was supported by EPSRC Grant EP/H00534X/1.

\renewcommand{\refname}{Reference}
\bibliographystyle{amsalpha}
\bibliography{converse}

\providecommand{\bysame}{\leavevmode\hbox to3em{\hrulefill}\thinspace}
\providecommand{\MR}{\relax\ifhmode\unskip\space\fi MR }
\providecommand{\MRhref}[2]{%
  \href{http://www.ams.org/mathscinet-getitem?mr=#1}{#2}
}
\providecommand{\href}[2]{#2}
\begin{thebibliography}{BHK98}

\bibitem[AL]{AL13}
Moshe Adrian and Baiying Liu, \emph{The local langlands correspondence for
  epipelagic supercuspidal representations of $\rm {GL_n (F)}$}, arXiv:
  1310.2585, 2013.

\bibitem[BH99]{BH1999}
Colin~J. Bushnell and Guy Henniart, \emph{Local tame lifting for {${\rm
  GL}(n)$}. {II}. {W}ildly ramified supercuspidals}, Ast\'erisque (1999),
  no.~254, vi+105. \MR{1685898 (2000d:11147)}

\bibitem[BH05]{BH2005}
\bysame, \emph{The essentially tame local {L}anglands correspondence. {I}}, J.
  Amer. Math. Soc. \textbf{18} (2005), no.~3, 685--710. \MR{2138141
  (2006a:22014)}

\bibitem[BH06]{BH2006}
\bysame, \emph{The local {L}anglands conjecture for {$\rm GL(2)$}}, Grundlehren
  der Mathematischen Wissenschaften [Fundamental Principles of Mathematical
  Sciences], vol. 335, Springer-Verlag, Berlin, 2006. \MR{2234120
  (2007m:22013)}

\bibitem[BHK98]{BHK1998}
Colin~J. Bushnell, Guy~M. Henniart, and Philip~C. Kutzko, \emph{Local
  {R}ankin-{S}elberg convolutions for {${\rm GL}_n$}: explicit conductor
  formula}, J. Amer. Math. Soc. \textbf{11} (1998), no.~3, 703--730.
  \MR{1606410 (99h:22022)}

\bibitem[BK93]{BK1993}
Colin~J. Bushnell and Philip~C. Kutzko, \emph{The admissible dual of {${\rm
  GL}(N)$} via compact open subgroups}, Annals of Mathematics Studies, vol.
  129, Princeton University Press, Princeton, NJ, 1993. \MR{1204652
  (94h:22007)}

\bibitem[Bus87]{Bus1987}
Colin~J. Bushnell, \emph{Hereditary orders, {G}auss sums and supercuspidal
  representations of {${\rm GL}_N$}}, J. Reine Angew. Math. \textbf{375/376}
  (1987), 184--210. \MR{882297 (88e:22024)}

\bibitem[CPS94]{CPS1994}
James~W. Cogdell and IIya Piatetski-Shapiro, \emph{Converse theorems for {${\rm
  GL}_n$}}, Inst. Hautes \'Etudes Sci. Publ. Math. (1994), no.~79, 157--214.
  \MR{1307299 (95m:22009)}

\bibitem[Has80]{Has1980}
Helmut Hasse, \emph{Number theory}, Grundlehren der Mathematischen
  Wissenschaften [Fundamental Principles of Mathematical Sciences], vol. 229,
  Springer-Verlag, Berlin, 1980, Translated from the third German edition and
  with a preface by Horst G{\"u}nter Zimmer. \MR{562104 (81c:12001b)}

\bibitem[JS85]{JS1985}
Herv{\'e} Jacquet and Joseph Shalika, \emph{A lemma on highly ramified
  {$\epsilon$}-factors}, Math. Ann. \textbf{271} (1985), no.~3, 319--332.
  \MR{787183 (87i:22048)}

\bibitem[Kud94]{Kudla1991}
Stephen~S. Kudla, \emph{The local {L}anglands correspondence: the
  non-{A}rchimedean case}, Motives ({S}eattle, {WA}, 1991), Proc. Sympos. Pure
  Math., vol.~55, Amer. Math. Soc., Providence, RI, 1994, pp.~365--391.
  \MR{1265559 (95d:11065)}

\end{thebibliography}
\nocite{Bus1987}

\texttt{SCHOOL OF MATHEMATICS, UNIVERSITY OF EAST ANGLIA, NORWICH, NR4 7TJ, UK}

\emph{E-mail address}: \texttt{xupeng2012@gmail.com}

\end{document}